\documentclass[12pt]{amsart}

\usepackage{graphicx,amssymb,amsmath,hyperref}
\usepackage{wasysym,color}

\newtheorem{theorem}{Theorem}
\newtheorem{lemma}[theorem]{Lemma}
\newtheorem{cor}[theorem]{Corollary}

\theoremstyle{definition}

\newtheorem{example}[theorem]{Example}

\newtheorem*{remark*}{Remark}

\title{Mixed Volume Techniques for Embeddings of Laman Graphs}
\author{Reinhard Steffens, Thorsten Theobald}
\index{Steffens, Reinhard} \index{Theobald, Thorsten}
\thanks{ Research supported through DFG grant TH1333/1-1}
\address{Goethe-Universit\"at, FB 12 -- Institut f\"ur Mathematik, Postfach 111932, D-60054 Frankfurt~ a.M.}

\newcommand{\C}{{\mathbb C}}
\newcommand{\R}{{\mathbb R}}
\newcommand{\Q}{{\mathbb Q}}

\newcommand{\dnull}{{\bf 0 }}
\newcommand{\deins}{{\bf 1 }}
\DeclareMathOperator{\conv}{conv}
\DeclareMathOperator{\NP}{NP}
\DeclareMathOperator{\MV}{MV}

\definecolor{RED}{rgb}{0.6,0,0}

\begin{document}

\maketitle

\begin{abstract}
Determining the number of embeddings of Laman graph frameworks is an open problem
which corresponds to understanding the solutions of the resulting systems of 
equations.
In this paper we investigate the bounds which can be obtained from the viewpoint
of Bernstein's Theorem. 
The focus of the paper is to provide the methods to study the mixed volume of suitable 
systems of  polynomial equations obtained from the edge length constraints.
While in most cases the resulting bounds are weaker than the best known bounds on the 
number of embeddings, for some classes of graphs the bounds are tight.

\end{abstract}
{\bf Keywords:} {\it Mixed volume, Laman graphs, minimally rigid graphs, Bernstein's Theorem, BKK theory}

\section{Introduction}
Let $G=(V,E)$ be a graph on $n$ vertices with $2n-3$ edges. If each
subset of $k$ vertices spans at most $2k-3$ edges, we say that $G$ has the
{\it Laman property} and call it a {\it Laman graph} (see \cite{Laman}). 
A {\it framework} is a tuple $(G,L)$ where $G=(V,E)$ is a graph and $L=\{l_{i,j}\, :\, [v_i,v_j]\in E\}$ is a set of $|E|$ positive numbers interpreted as edge lengths.
For generic edge lengths, Laman graph frameworks are minimally rigid (see \cite{Connelly}), 
i.e. they are rigid and they become flexible if any edge is removed. 

A {\it Henneberg sequence} for a graph $G$ is a sequence
$(G_i)_{3\leq i \leq r}$ of Laman graphs such
that $G_3$ is a triangle, $G_r=G$ and each $G_i$ is obtained by
$G_{i-1}$ via one of the following two types of steps: A {\it
  Henneberg I step} adds one new vertex $v_{i+1}$ and two new edges,
connecting $v_{i+1}$ to two arbitrary vertices of $G_i$. A {\it
  Henneberg II step} adds one new vertex  $v_{i+1}$ and three new edges,
connecting $v_{i+1}$ to three vertices of $G_i$ such that at least two
of these vertices are connected via an edge $e$ of $G_i$ and this certain
edge $e$ is removed (see Figure \ref{FigHenneberg}).
\begin{figure}[h]
\setlength{\unitlength}{0.18pt}
	\begin{center}
		\begin{picture}(1200,550)(0,0)
			\put(10,40){\includegraphics[scale=0.18]{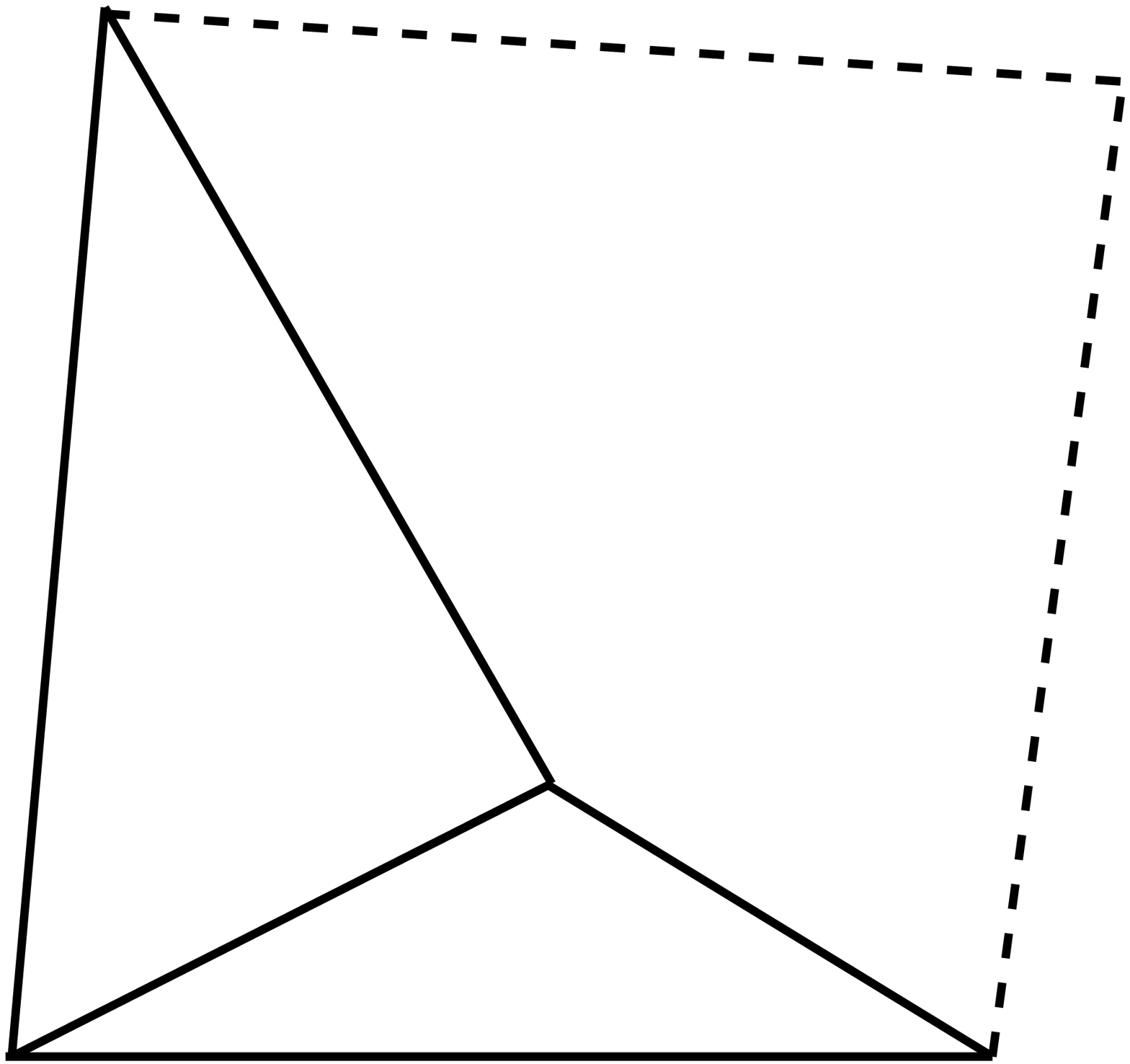}}
			\put(650,40){\includegraphics[scale=0.18]{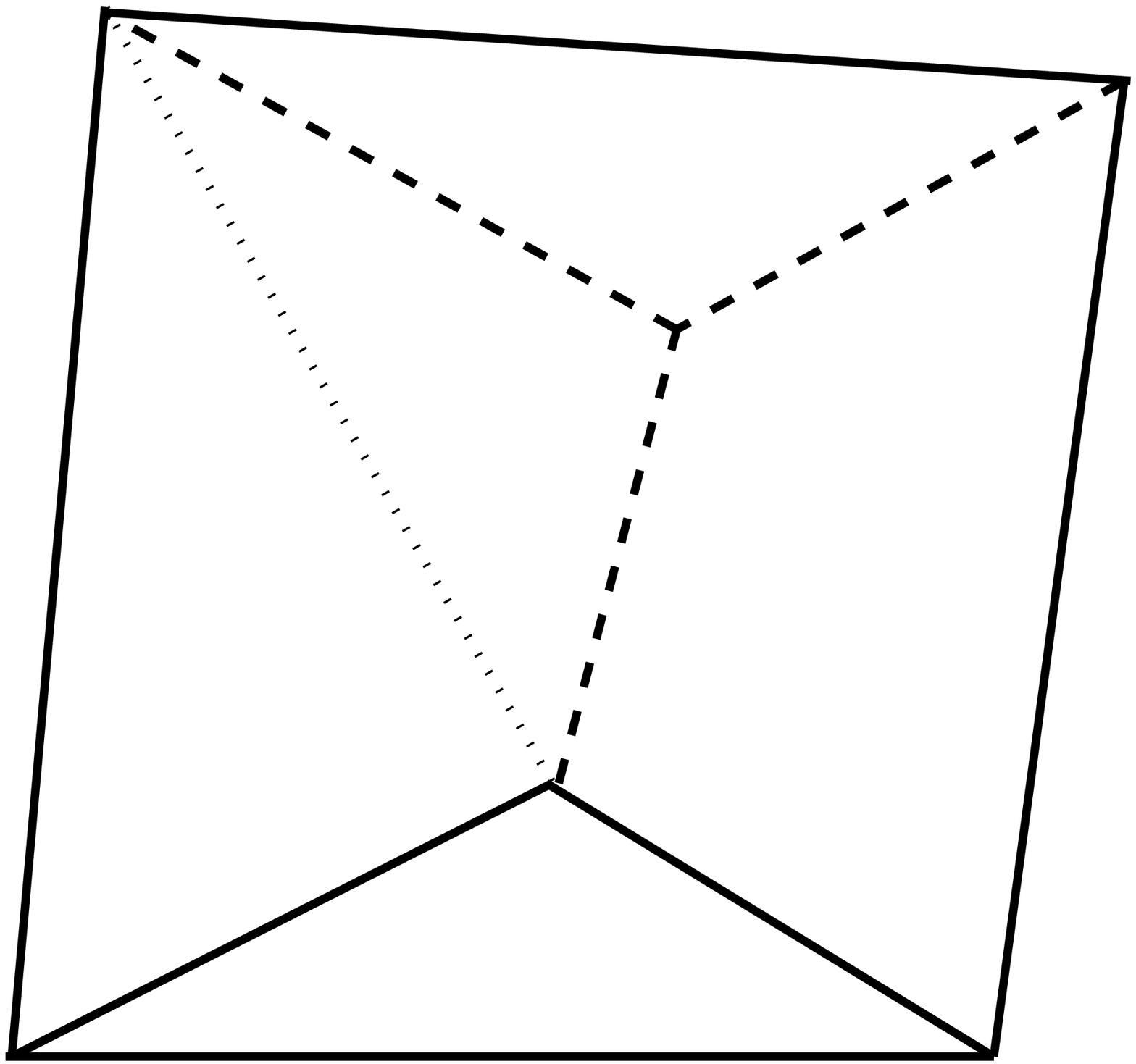}}
			\put(0,0){\mbox{$v_1$}}
			\put(450,0){\mbox{$v_2$}}
			\put(640,0){\mbox{$v_1$}}
			\put(1090,0){\mbox{$v_2$}}
			\put(290,165){\mbox{$v_3$}}
			\put(940,165){\mbox{$v_3$}}
			\put(0,540){\mbox{$v_4$}}
			\put(650,545){\mbox{$v_4$}}
			\put(520,510){\mbox{$v_5$}}
			\put(1170,510){\mbox{$v_5$}}
			\put(980,340){\mbox{$v_6$}}
		\end{picture}
	\caption{A Henneberg I and a Henneberg II step. New edges are dashed and the deleted edge is pointed.}
	\label{FigHenneberg}
\end{center}
\end{figure}
\setlength{\unitlength}{1pt}

Any Laman graph $G$ can be constructed via a Henneberg sequence and
any graph constructed via a Henneberg sequence has the Laman property (see
\cite{StreinuTheran,TayWhiteley}). We call $G$ a {\it Henneberg I graph} if it
is constructable using only Henneberg I steps. Otherwise we call it
{\it Henneberg II}.

 Given a Laman graph framework we want to know how many embeddings, i.e. maps $\alpha: V \rightarrow \R^2$, exist such that the Euclidean distance between two points in the image is exactly $l_{i,j}$ for all $[v_i,v_j]\in E$. Since every rotation or translation of an embedding gives another one, we ask how many embeddings exist {\it modulo rigid motions}.
Due to the minimal rigidity property, questions about embeddings of Laman graphs arise naturally in
rigidity and linkage problems (see \cite{Haas,Thorpe}).
Graphs with less edges will have zero or infinitely many embeddings modulo rigid motions,
and graphs with more edges do not have any embeddings for a generic choice of edge lengths.

Determining the maximal number of embeddings (modulo rigid motions)
for a given Laman graph is an open problem.
The best upper bounds are due to Borcea and Streinu (see \cite{Borcea,BorceaStreinu})
who show that the number of embeddings is bounded by $\binom{2n-4}{n-2} \approx \frac{4^{n-2}}{\sqrt{n-2}}$. Their bounds are based on
degree results of determinantal varieties. 

A general method to study the number of (complex) solutions of systems of
polynomial equations is to use Bernstein's Theorem \cite{Bernstein} for sparse
polynomial systems. This theorem provides bounds on the number of solutions
in terms of the mixed volume of the underlying Newton polytopes.
Since the systems of polynomial
equations describing the Laman embeddings are sparse, the question arose
how good these Bernstein bounds are for the Laman embedding problem.
While for concrete systems of equations, the mixed volume can be
computed algorithmically, studying the mixed volume for \emph{classes
of polytopes} is connected with a variety of issues in convex geometry
(such as understanding the Minkowski sum of the polytopes).

In this paper, we study the quality of the Bernstein bound on the Laman embedding problem and provide methods to handle the resulting convex geometric problems.
In most cases, our bounds are worse than the bounds in \cite{BorceaStreinu}.
However, we think that the general methodology of studying Bernstein
bounds nevertheless provides an interesting technique, and we see
the main contribution of this paper in providing the technical tools
(such as achieving to determine the mixed volume)
to compute these bounds for whole classes of graphs.
It is particularly 
interesting that for some classes of graphs, the mixed volume
bound is tight.

To use these algebraic tools for the embedding problem 
we formulate that problem as a 
system of polynomial equations in the $2n$ unknowns $(x_1,y_1,\dots,x_n,y_n)$ where $(x_i,y_i)$ denote the coordinates of the embedding of the vertex $v_i$. 
Each prescribed edge length translates into a polynomial equation. I.e. if $e_k:=[v_i,v_j]\in E$ with length $l_{i,j}$, we require $h_k(x):=(x_i-x_j)^2+(y_i-y_j)^2-l_{i,j}^2=0$. Thus we obtain a system of $|E|$ quadratic equations whose solutions represent the embeddings of our framework. To get rid of translations and rotations we fix one point $(x_1,y_1)=(c_1,c_2)$ and the direction of the embedding of the edge $[v_1,v_2]$ by setting $y_2=c_3$. (Here we assume without loss of generality that there is an edge between $v_1$ and $v_2$.) For practical reasons we choose $c_i \neq 0$ and as well $c_1\neq l_{1,2}$. 
Hence we want to study the solutions to the following system of $2n$ equations.
\begin{equation}\label{SoE}
\left. \begin{cases} h_1(x):=x_1-c_1=0 \\ h_2(x):=y_1-c_2=0 \\ h_3(x):=x_2-(l_{1,2}-c_1)=0 \\ h_4(x):=y_2-c_3=0 \\ h_k(x):=(x_i-x_j)^2+(y_i-y_j)^2-l_{i,j}^2 =0 \quad \forall e_k=[v_i,v_j]\in E-\{[v_1,v_2]\}   \end{cases} \right\}
\end{equation}

The paper is structured as follows. In Section~\ref{sec:Definitions} we review the concepts of mixed volumes and Bernstein's Theorem. In Section~\ref{sec:Tools} we present some technical tools to simplify mixed volume calculation. Then, in Section~\ref{sec:BKK} we discuss the quality of the Bernstein bounds on the Laman embedding problem. 

\section{Preliminaries}\label{sec:Definitions}

\subsection{Mixed volumes and mixed subdivisions}\label{sec:MixedVolumes}
The \emph{Minkowski sum} of two sets $A_1, A_2 \subset \R^k$ is defined as
\[
A_1+A_2 = \left\{a_1+a_2 \, |\, a_1\in A_1, a_2\in A_2 \right\} \ . 
\]
Let $P_1,\dots,P_k$ be $k$ polytopes in $\R^k$. For non-negative parameters 
$\lambda_1,\dots,\lambda_k$ the function
$\text{vol}_k(\lambda_1P_1+\dots+\lambda_kP_k)$ is a homogeneous
polynomial of degree $k$ in $\lambda_1,\dots,\lambda_k$ with
non-negative coefficients (see e.g. \cite{Schneider, Webster}).  The coefficient of the mixed
monomial $\lambda_1\cdots\lambda_k$ is called the {\it mixed volume of $P_1,\dots,P_k$} and is
denoted by $\MV_k(P_1,\dots,P_k)$. 

We denote by $\textup{MV}_{k}(P_1,d_1; \dots; P_r,d_r)$ the mixed volume where $P_i$ is taken $d_i$ times and $\sum_{i=1}^r d_i =k$.
The mixed volume is invariant under permutation of its arguments, it is linear in each argument, i.e.
\begin{equation}
\MV_k(\dots, \alpha P_i + \beta P'_i,\dots)= 
\alpha\, \MV_k(\dots, P_i ,\dots)+ \beta \, \MV_k(\dots, P'_i ,\dots) \label{linearity} 
\end{equation}
and it generalizes the usual volume in the sense that
\begin{equation} \label{MV_Vol}
\MV_k(P,\dots,P)= k! \,\text{vol}_k (P)
\end{equation}
holds (see \cite{Schneider}).

Let $P=P_1+\dots+P_r\subset\R^k$ be a Minkowski sum of polytopes that affinely spans $\R^k$. A sum $C=F_1+\dots+F_r$ of faces $F_i\subset P_i$ is called {\it cell} of $P$. A {\it subdivision} of $P$ is a collection $\Gamma=\{C_1,\dots,C_m\}$ of cells such that each cell is of full dimension, the intersection of two cells is a face of both and the union of all cells covers $P$. Each cell is given a type $type(C)= (\text{dim}(F_1),\dots, \text{dim}(F_r))$. Clearly the entries in the type vector sum up to at least the dimension of the cell $C$.  A subdivision is called {\it mixed} if for each cell $C\in \Gamma$ we have that $\sum d_i =k$ where $type(C)=(d_1,\dots,d_r)$. 
Cells of type $(d_1,\dots,d_r)$ with $d_i\geq 1$ for each $i$ will be called {\it mixed cells}.

With this terminology the mixed volume can be calculated by
\begin{equation}\label{explicitMV_2}
\MV_k(P_1,d_1;\dots;P_r,d_r) =\sum_{C} d_1!\, \cdots d_r!\ \text{vol}_k\,(C) 
\end{equation}
where the sum is over all cells $C$ of type $(d_1,\dots,d_r)$ in an arbitrary mixed subdivision of $P_1+\dots+P_r$ (see \cite{Huber}).

To construct mixed subdivisions we proceed as in \cite{Huber}. Not every subdivision can be constructed in this way but since we only need one arbitrary mixed subdivision this simple construction can be used. For each polytope $P_i$ choose a linear lifting function $\mu_i:\R^k \rightarrow \R$ identified by an element of $\R^k$. By $\hat{P_i}$ we denote the lifted polytopes $\conv \{(q,\langle \mu_i,q \rangle)\, : \, q\in P_i \}\subset \R^{k+1}$, where $\langle \cdot, \cdot \rangle$ denotes the Euclidean scalar product.

The set of those facets of $\hat{P}:=\hat{P}_1+\dots+\hat{P}_r$ which have an inward pointing normal with a positive last coordinate is called \emph{the lower hull} of $\hat{P}$. Projecting down this lower hull back to $\R^k$ by forgetting the last coordinate yields a subdivision of $P_1+\dots+P_r$. Such a subdivision is called {\it coherent} and is said to be {\it induced by} $\mu =(\mu_1,\dots,\mu_r)$.

\begin{example}\label{Example:1}
Let 
\[
P=\conv\left\{\begin{pmatrix}0\\ 0 \end{pmatrix}, \begin{pmatrix}3\\ 0 \end{pmatrix},\begin{pmatrix}0\\ 2 \end{pmatrix},\begin{pmatrix}3\\ 2 \end{pmatrix} \right\}\, , \quad Q=\conv\left\{\begin{pmatrix}1\\ 0 \end{pmatrix}, \begin{pmatrix} 0 \\ \frac{3}{2}\end{pmatrix},\begin{pmatrix}3\\ 3 \end{pmatrix} \right\}\ .
\]
The Minkowski sum of $P$ and $Q$ is depicted in Figure~\ref{Figure:ExMinkSum} together with one of the possible coherent mixed subdivisions.
\begin{figure}[ht]
\setlength{\unitlength}{0.38pt}
	\begin{center}
		\begin{picture}(650,300)(0,0)
			\put(0,0){\includegraphics[scale=0.38]{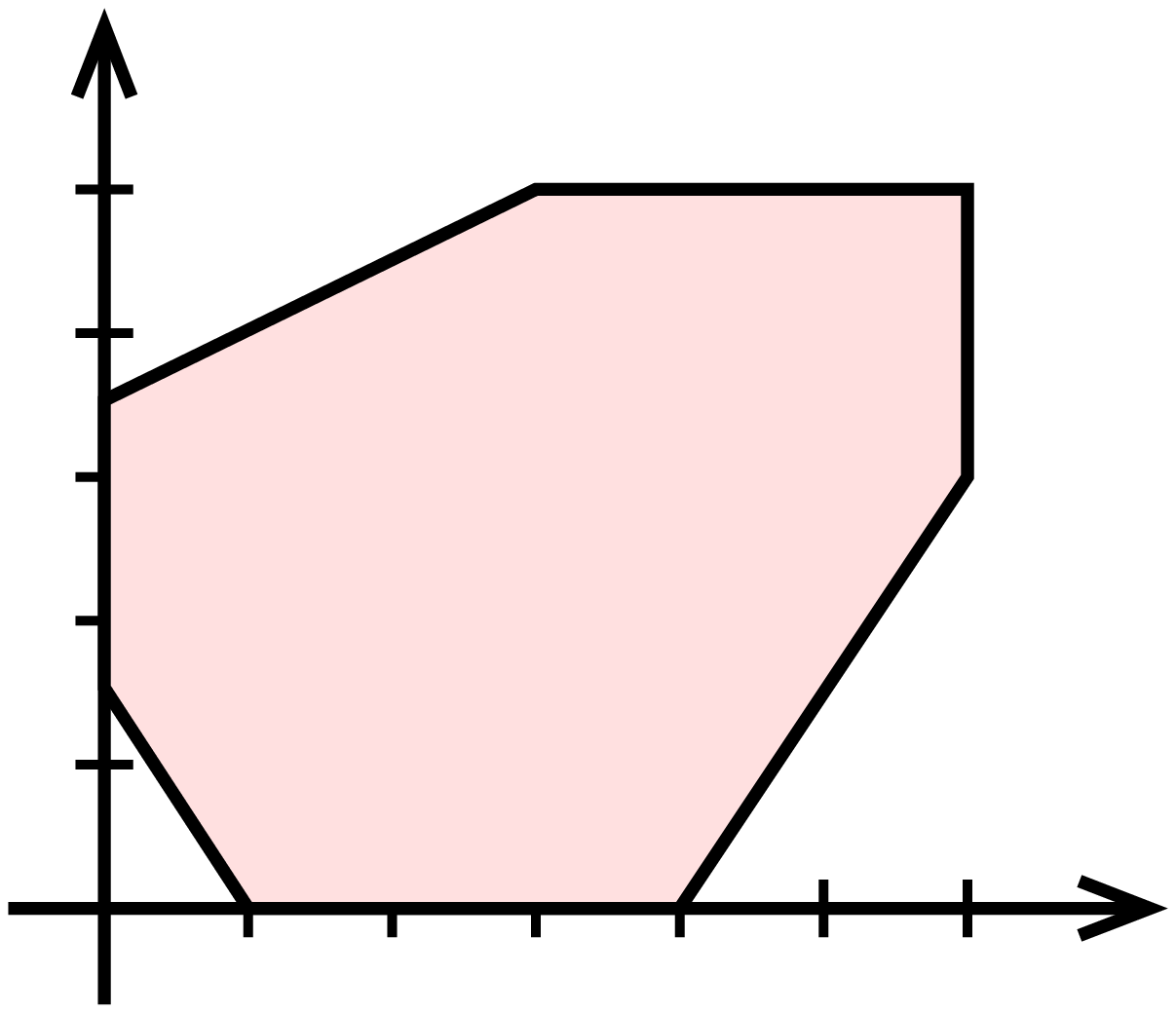} }
			\put(390,30){\includegraphics[scale=0.38]{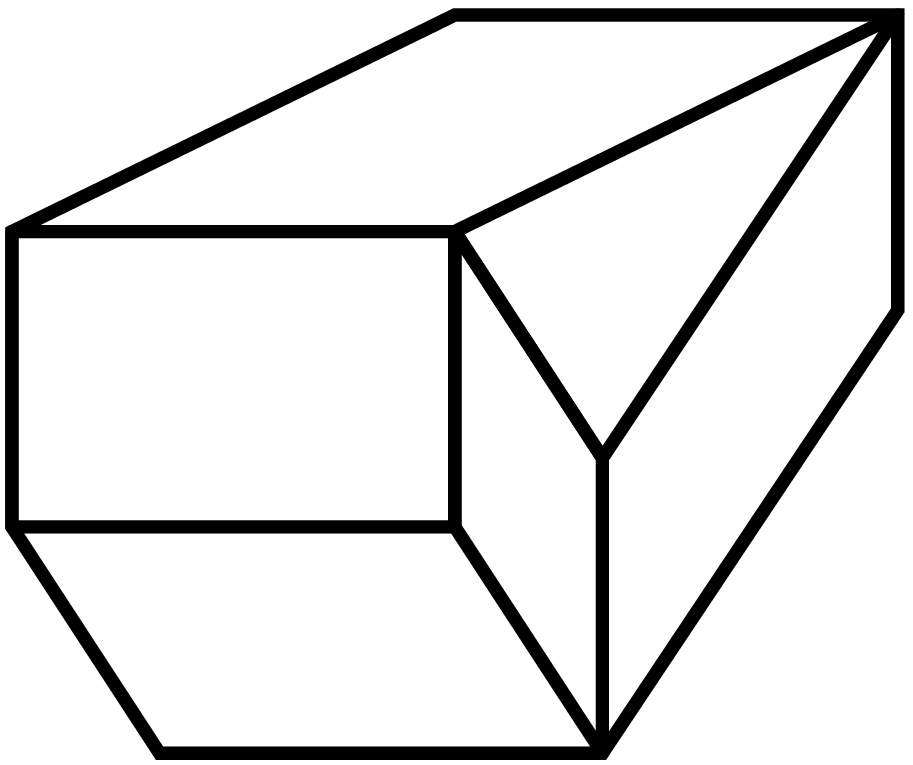}}
			\put(100,120){\mbox{$P+Q$}}
			\put(440,125){\mbox{$P$}}
			\put(555,165){\mbox{$Q$}}
			\put(500,200){\mbox{$C_1$}}
			\put(524,100){\mbox{$C_2$}}
			\put(595,140){\mbox{$C_3$}}
			\put(460,55){\mbox{$C_4$}}
		\end{picture}
	\caption{Left: The Minkowski sum of $P$ and $Q$. Right: A mixed subdivision $\Gamma$ of $P+Q$.}
	\label{Figure:ExMinkSum}
\end{center}
\end{figure}
\setlength{\unitlength}{1pt}
\end{example}

\subsection{BKK theory} \label{sec:Bernstein}
The main tool in this work is the following theorem that provides a connection between solutions to systems of polynomial equations and discrete geometry. For a polynomial $f=\sum_{\alpha\in A} c_{\alpha}x^{\alpha}\in \C[x_1,\dots,x_k]$ the Newton polytope $\NP(f)\subset\R^k$ is the convex hull of the monomial exponent vectors, i.e. $\NP(f)=\conv A$. Let $\C^* := \C \setminus \{0\}$.
\begin{theorem}[Bernstein \cite{Bernstein}] \label{The:Bernstein}
Given polynomials $f_1,\dots,f_k\in\C[x_1,\dots,x_k]$  with finitely
many common zeroes in $(\C^*)^k$ and let $\NP(f_i)$ denote the Newton polytope of
$f_i$. Then the number of common zeroes of the $f_i$ in
$(\C^*)^k$ is bounded above by the mixed volume
$\textup{MV}_k(\NP(f_1),\dots,\NP(f_k))$. Moreover for generic choices of the
coefficients in the  $f_i$, the number of common solutions is exactly
$\textup{MV}_k(\NP(f_1),\dots,\NP(f_k))$. 
\end{theorem}

Various attempts have been made to generalize these results to count all common roots in $\C^k$ (see for example \cite{EmirisVerschelde, HuberSturmfels2, LiWang}). The easiest, but sometimes not the best bound is $\MV_k(\text{conv}(\NP(f_1)\cup {0}),\dots,\text{conv}(\NP(f_k)\cup{0}))$ which is shown in \cite{LiWang}. Since the Newton polytopes of system (\ref{SoE}) all contain the point $0$ as a vertex, the mixed volume of (\ref{SoE}) yields a bound on the number of solutions in $\C$ rather then only on those in $\C^*$.  

The bound on the number of solutions of a polynomial system arising from Bernstein's Theorem is also often referred to as the {\it BKK bound} due to the work of Bernstein, Khovanskii and Kushnirenko. The BKK bound generalizes the B\'ezout bound (see \cite[Chapter 7]{CLO2}) and for sparse polynomial systems it is often significantly better. 

Bernstein also gives an explicit condition when a choice of coefficients is generic.
Let $w$ be a non-zero vector and let $\partial_w P$ denote the face of a polytope $P$ which is minimal with respect to the direction $w$. Also we set $\partial_w f = \sum_{\alpha \in \partial_w \NP(f)}c_{\alpha} x^{\alpha}$ to be the face equation with respect to  $w$. 
\begin{theorem}[Bernstein's Second Theorem \cite{Bernstein}] \label{The:Bernstein2}
If for all $w\neq 0$, the face system $\partial_w f_1 =0, \dots ,\partial_w f_k =0  $ has no solution in $(\mathbb{C}^*)^k$, then the mixed volume of the Newton polytopes of the $f_i$ gives the exact number of common zeros in $(\mathbb{C}^*)^k$ and all solutions are isolated. Otherwise it is a strict upper bound.
\end{theorem} 
Note that it is necessary for a direction $w$ to be a witness of the degeneracy that it lies on the tropical prevariety (see \cite{FirstStepsTropical}) of the polynomials $f_1,\dots,f_k$.

\section{New technical tools to simplify mixed volume calculation} \label{sec:Tools}
In the special case of Henneberg I graphs  system (\ref{SoE}) is of a shape that allows to separate the mixed volume calculation into smaller pieces. 
The main tool to do this is the following Lemma.
An equivalent decomposition result was already mentioned in \cite{Burago} in which the authors refer to \cite{Fedotov} (in Russian) for the proof. For the convenience of the reader we provide here a proof based
on the properties of symmetric multilinear functions.

\begin{lemma} 
\label{SeparationLemma}
Let $P_1,\dots,P_k$ be polytopes in $\R^{m+k}$ and
$Q_1,\dots,Q_m$ be polytopes in $\R^m\subset\R^{m+k}$ . 
Then
\begin{equation} \label{eq:SepLemma}
\textup{MV}_{m+k}(Q_1,\dots,Q_m,P_1,\dots,P_k)=
\textup{MV}_{m}(Q_1,\dots,Q_m)\cdot\textup{MV}_{k}(\pi(P_1),\dots,\pi(P_k))
\end{equation}
where  $\pi: \R^{m+k} \rightarrow \R^k$ denotes the projection on the
last $k$ coordinates.
\end{lemma}
\begin{proof}
First we show the Lemma in the \emph{semimixed case} where $Q_1=\dots=Q_m=:Q$ and $P_1=\dots=P_k=:P$, then
we use properties of symmetric multilinear functions to reduce the general situation to the
semimixed case.

By (\ref{MV_Vol}) we have to show first that 
\begin{equation}\label{Case1}
\textup{MV}_{m+k}(Q,\dots,Q,P,\dots,P)=m!\, k!\ \textup{vol}_m(Q)\cdot\textup{vol}_k(\pi(P)) 
\end{equation}
where $Q$ is taken $m$ times and $P$ is taken $k$ times. But this formula for semimixed systems is a special case of Lemma 4.9 in \cite{Ewald} or also of Theorem 1 in \cite{Betke}.

Let $\mathcal{P}^m$ (resp. $\mathcal{P}^{m+k}$)  be the set of all $m$-dimensional (resp. $(m+k)$-dimensional)  polytopes and define two functions $g_1$ and $g_2$ on $(\mathcal{P}^m)^m\times (\mathcal{P}^{m+k})^k$ via 
\begin{eqnarray*}
g_1(Q_1,\dots,Q_m,P_1,\dots,P_k) &:=& \textup{MV}_{m+k}(Q_1,\dots,Q_m,P_1,\dots,P_k)\\
g_2(Q_1,\dots,Q_m,P_1,\dots,P_k) &:=& \textup{MV}_{m}(Q_1,\dots,Q_m)\cdot\textup{MV}_{k}(\pi(P_1),\dots,\pi(P_k))\ .
\end{eqnarray*}
Due to the properties of mixed volumes (see Paragraph~\ref{sec:MixedVolumes}) it is easy to see that $g_1$ and $g_2$ are invariant under changing the order of the $Q_i$ and under changing the order of the $P_j$. Furthermore it follows from (\ref{linearity}) that both functions are linear in each argument. 

Hence, for fixed $P_1, \dots, P_k$ the induced mappings
\[
\tilde{g}_i^{(P_1,\dots,P_k)}(Q_1,\dots,Q_m) :=  g_i(Q_1,\dots,Q_m,P_1,\dots,P_k)  \qquad (i=1,2)
\]
are symmetric and multilinear, and analogously, for fixed $Q$, the mappings
\[
\bar{g}^{(Q)}_i(P_1,\dots,P_k) :=  g_i(Q,\dots,Q,P_1,\dots,P_k) \qquad (i=1,2)
\]
are symmetric and multilinear.
For any semigroups $A,B$ and any symmetric multilinear function
$f: A^n \rightarrow B$, it follows from an inclusion-exclusion argument (see \cite[Theorem 3.7]{Ewald}) that
\begin{equation} \label{multilinear}
f(a_1,\dots,a_n)= \frac{1}{n!}\sum_{1\leq i_1< \cdots <i_q\leq n}(-1)^{n-q} f(a_{i_1}+ \cdots+a_{i_q},\dots,a_{i_1}+\cdots+a_{i_q})\ .
\end{equation}
Hence we have for $i=1,2$ that
\begin{align*}
& g_i(Q_1,\dots,Q_m,P_1,\dots,P_k) \\
=\ &\tilde{g}_i^{(P_1,\dots,P_k)}(Q_1,\dots,Q_m) \\
=\ &\frac{1}{m!}\sum_{1\leq i_1<\cdots<i_q\leq m} (-1)^{m-q} \ 
                         \tilde{g}_i^{(P_1,\dots,P_k)} (Q_{i_1}+\cdots+Q_{i_q},\dots,Q_{i_1}+\cdots+Q_{i_q}) \\
=\ &\frac{1}{m!}\sum_{1\leq i_1<\cdots<i_q\leq m}(-1)^{m-q}\  
                          \bar{g}^{(Q_{i_1}+\cdots+Q_{i_q})}_i(P_1,\dots,P_k) \ .
\end{align*}
Since we can expand $\bar{g}^{(Q_{i_1}+\cdots+Q_{i_q})}_i(P_1,\dots,P_k)$ by using (\ref{multilinear}) as well, we see that both functions $g_1$ and $g_2$ are fully determined by their images of tuples of polytopes where  $Q_1=\cdots=Q_m=Q$ and $P_1=\cdots=P_k=P$. This proves the Lemma.
\end{proof}

Another technical tool which is employed in a subsequent proof is the following Lemma. This goes back to an idea of Emiris and Canny \cite {EmirisCanny} to use linear programming and the formula \ref{explicitMV_2} to compute the mixed volume. 
\begin{lemma} \label{LiftingLemma}
Given polytopes $P_1,\dots,P_k$ $\subset \R^k$ and lifting vectors $\mu_1,\dots,\mu_k\in\R^k_{\geq 0}$. Denote the vertices of $P_i$ by $v^{(i)}_1,\dots,v^{(i)}_{r_i}$ and choose one edge $e_i=[v^{(i)}_{t_i}, v^{(i)}_{l_i}]$ from each $P_i$. Then 
$C:=e_1+\dots+e_k$ is a mixed cell of the mixed subdivision induced by the liftings $\mu_i$ if and only if
\begin{enumerate}
\item[i)] The edge matrix $E:=V_a - V_b$ is non-singular (where $V_a:=(v^{(1)}_{t_1},\dots,v^{(k)}_{t_k})$ and $V_b:=(v^{(1)}_{l_1},\dots,v^{(k)}_{l_k})$) \ and 
\item[ii)] For all polytopes $P_i$ and  all vertices $v^{(i)}_s$ of $P_i$ which are not in $e_i$ we have:
\begin{equation} 
\left( \langle \mu_1-\mu_i,\vec{e_1}\rangle,\dots,\langle\mu_k-\mu_i,\vec{e_k}\rangle \right)\cdot E^{-1} \cdot \left(v^{(i)}_{l_i}-v^{(i)}_s \right) \geq 0 \label{LiftingCondition}
\end{equation}
where $\vec{e_i}=v^{(i)}_{t_i}-v^{(i)}_{l_i}$.
\end{enumerate}
\end{lemma}
Before beginning with the proof we start with some auxiliary considerations about how to apply linear programming (\emph{LP}) here. In \cite{EmirisCanny} it is shown that the test, if a cell lies on the lower envelope of the lifted Minkowski sum can be formulated as a linear program.  Let $\hat{m}_i\in \R^{k+1}$ denote the midpoint of the lifted edge $\hat{e}_i$ of $\hat{P}_i$ such that $\hat{m}=\hat{m}_1+\dots + \hat{m}_k$ is an interior point of the Minkowski sum $\hat{e}_1+\dots+\hat{e}_k$. Consider the linear program 
\begin{align}
\text{maximize } s &\in \R_{\geq 0} \label{LP} \\ 
\text{s.t. } \hat{m} &- (0,\dots,0,s) \in \hat{P}_1+\dots+\hat{P}_k\nonumber \ .
\end{align}
If we denote the vertices of $P_i$ by $v^{(i)}_{1},\dots,v^{(i)}_{r_i}$ this can be written as
\begin{align*}
\text{maximize } s  & \in \R_{\geq 0} \\
\text{s.t. }      \hat{m}-(0,\dots,0,s) &= \displaystyle\sum_{i=1}^{k} \sum_{j=1}^{r_i} \lambda^{(i)}_{j} \hat{v}^{(i)}_{j}  \\
          \sum_{j=1}^{r_i} \lambda^{(i)}_{j} &=1 \quad \forall\, i=1,\dots,n  \\
          \lambda^{(i)}_{j} &\geq 0 \quad \forall\, i,j   \ .\\
\end{align*}
$s$ measures the distance of $\hat{m}$ to the lower envelope of the Minkowski sum. Hence $\hat{m}$ lies on the lower envelope of $\hat{P}_1+\dots+\hat{P}_k$ if and only if the optimal value of (\ref{LP}) is zero.

In standard matrix form, 
the linear program (\ref{LP}) can be written as $\max \{c^T x\, : \, Ax=b, x \ge 0 \}$ with 
\begin{eqnarray*}
A&=& \begin{pmatrix}
     v^{(1)}_{1} & \dots & v^{(1)}_{r_1}  & \dots & \dots &  v^{(k)}_{1} & \dots & v^{(k)}_{r_{k}} & \dnull_{k}  \\[0.1cm]
\langle \mu_1,v^{(1)}_{1} \rangle & \dots &\langle \mu_1,v^{(1)}_{r_1} \rangle&\dots& \dots &\langle \mu_k,v^{(k)}_{1} \rangle& \dots &\langle \mu_k,v^{(k)}_{r_k} \rangle& 1     \\[0.1cm]
                & \deins^T_{r_1} &         & \dnull^T_{r_2} & \dots &                & \dnull^T_{r_{k}} &             &0 \\[0.1cm]
                & \dnull^T_{r_1} &         & \deins^T_{r_2} & \dots &                & \dnull^T_{r_{k}} &             &0 \\[0.1cm]
                & \vdots         &         &                & \ddots &                & \vdots            &            &\vdots \\[0.1cm]
                & \dnull^T_{r_1} &         & \dnull^T_{r_2} & \dots &                & \deins^T_{r_{kn}}  &            &0 \\[0.1cm]
     \end{pmatrix} \, , \\
b^T&=&(\hat{m}, \deins^T_{k})\ \in \R^{2k+1} \, , \\
c^T&=&(\dnull^T_{r_1+\dots +r_k},1) \ \in \R^{r_1+\dots+r_k+1}
\end{eqnarray*}
and variables 
$x^T = (\lambda^{(1)}_{1},\dots,\lambda^{(1)}_{r_1}, \dots \dots ,\lambda^{(k)}_{1},\dots,\lambda^{(k)}_{r_{k}},s)\ \in \R^{r_1+\dots+r_k+1}$.
Here $\dnull_{k}$ and $\deins_k$ denote the all-0-vector and the all-1-vector in $\R^k$, respectively.
In this notation the point $\hat{m}$ from \eqref{LP} corresponds to $\bar{x}=(\lambda^{(1)}_{1},\dots ,\lambda^{(k)}_{r_{k}},s)$ where $s=0$ and $\lambda^{(i)}_{j}=\frac{1}{2}$ if the edge $\hat{e}_i$ contains the vertex $\hat{v}^{(i)}_j$ and $\lambda^{(i)}_{j}=0$ otherwise. 

Given a feasible vertex $\bar{x}\geq 0$ of the LP, let $B$ be a
(not necessarily unique) choice of columns of $A$ such that the submatrix $A_B$ consisting of these columns satisfies $A^{-1}_B \cdot b=\bar{x}$. Let $A_N$ be the submatrix of $A$ consisting of the remaining columns and define $c_B$ and $c_N$ in the same way. By LP duality (see, e.g. \cite{Groetschel})
$\bar{x}$ is optimal if and only if
\begin{equation}\label{SimplexCriterion}
c_N^T-c_B^T\cdot A^{-1}_B\cdot A_N \leq 0 \quad \text{(componentwise)} \ .
\end{equation}

To prove Lemma~\ref{LiftingLemma} we assume that $\bar{x}$ is optimal and deduce conditions on the lifting vectors $\mu_i$ by using the inequality \eqref{SimplexCriterion}.

\medskip\noindent
{\it Proof of Lemma \ref{LiftingLemma}.}
Note that $C$ is full-dimensional and hence has a non-zero volume if and only if  $E$ is non-singular. In the following only this case will be  considered. To simplify the notation write $\mu(V)$ to denote $(\langle \mu_1,v_1\rangle,\dots,\langle\mu_k,v_k\rangle)$.

We know that $C$ is a mixed cell if and only if the following $\bar{x}$ is the optimal solution to the linear program defined above:
\[
\bar{x}=(\lambda^{(1)}_{1},\dots,\lambda^{(k)}_{r_{k}},0) \text{ where } \lambda^{(i)}_{j}=\begin{cases} \frac{1}{2}, & j\in \{t_i, l_i\} \\ 0, & \text{else } \end{cases} \ .
\]
The submatrices of $A$ corresponding to $\bar{x}$ are 
\[
A_B=\begin{pmatrix}V_a                      &  V_b                     & \dnull_k \\
                   \mu(V_a)                 & \mu(V_b)                 & 1        \\
                   \textup{Id}_{k}           & \textup{Id}_{k}           & \dnull_k  
     \end{pmatrix} 
\quad\text{and}\quad
A_N=\begin{pmatrix}v^{(i)}_s\\ \mu_r\cdot v^{(i)}_s \\ \xi_i    \end{pmatrix}_{\begin{smallmatrix}1\leq i\leq k \\ 1\leq s\leq r_i \\ s\neq t_i,l_i\end{smallmatrix}} 
\]
where $\xi_i$ denotes the $i^{\text{th}}$ unit vector. Since 
\[
A_B^{-1}=\begin{pmatrix} E^{-1}          & \dnull_k                 & -E^{-1}\cdot V_b \\
                        -E^{-1}          & \dnull_k                 &  E^{-1}\cdot V_a \\
        -\mu(E)\cdot E^{-1}             & 1     & \mu(E)\cdot E^{-1}\cdot V_b-\mu(V_b)
         \end{pmatrix}
\]
and $c_N = (0,\dots,0) $ the criterion (\ref{SimplexCriterion}) implies that $\bar{x}$ is optimal if
and only if
\[
(0,\dots,0,1)\cdot A_B^{-1}\cdot A_N \geq 0 \ .
\]
But a single entry of the vector on the left can be explicitly computed as
\[
-\left(\mu(E)\cdot E^{-1}\right)\cdot v^{(i)}_s+\mu_r\cdot v^{(i)}_s+\left(\mu(E)\cdot E^{-1}\cdot V_b-\mu(V_b)\right)\cdot \xi_i
\]
which equals the left hand side of (\ref{LiftingCondition}).
\hfill $\Box$
\medskip

Note that (\ref{LiftingCondition}) is linear in the $\mu_j$. Hence, for a given a choice of edges 
this condition defines a cone of lifting vectors which induce 
a mixed subdivision that contains our chosen cell as a mixed cell.

\section{Application of the BKK theory on the graph embedding problem}\label{sec:BKK}

Our goal is to apply Bernstein's results to give bounds on the number of embeddings of
Laman graphs.
A first observation shows that for the formulation \eqref{SoE} the Bernstein bound is not
tight. Namely,
the system (\ref{SoE}) allows to choose a direction $w$ that satisfies the conditions of Bernstein's Second Theorem~\ref{The:Bernstein2}. The choice $w=(0,0,0,0,-1,-1,\dots,-1)$ yields the face system
\[
\left. \begin{cases} x_1-c_1=0 \\ y_1-c_2=0 \\ x_2-(l_{1,2}-c_1)=0 \\ y_2-c_3=0 \\ x_i^2+y_i^2=0 \quad \forall [v_1,v_i], [v_2,v_i] \in E \\ (x_i-x_j)^2+(y_i-y_j)^2=0\quad \forall [v_i,v_j]\in E \text{ with } i,j \neq 1,2   \end{cases} \right\}
\]
which has $(x_1,y_1,\dots,x_{n},y_{n})=(c_1,c_2,l_{1,2}-c_1,c_3,1,i,1,i,\dots,1,i)$ as a solution with non-zero complex entries. So the mixed volume of the system in (\ref{SoE}) is a strict upper bound on the number of graph embeddings.

To decrease this degeneracy we apply an idea of Ioannis Emiris\footnote{Personal communication at EuroCG 2008, Nancy} (see \cite{EmirisThesis}). Surprisingly the introduction of new variables for common subexpressions, which increases the B\'ezout bound, can decrease the BKK bound. Here we introduce for every  $i=1,\dots,n$ the variable 
$s_i$ together with the new equation $s_i=x_i^2+y_i^2$. This leads to the following system of equations.
\begin{equation}\label{SubSoE}
\left. \begin{cases} x_1-c_1=0 \\ y_1-c_2=0 \\ x_2-(l_{1,2}-c_1)=0 \\ y_2-c_3=0 \\ s_i+s_j-2x_ix_j-2y_iy_j-l_{i,j}^2 =0 \quad \forall [v_i,v_j]\in E-\{[v_1,v_2]\} \\
s_i -x_i^2 -y_i^2=0 \quad \forall i=1,\dots,n \end{cases} \right\}
\end{equation}
Experiments show that the system \eqref{SubSoE} is still not generic in the sense of Theorem~\ref{The:Bernstein2} for every underlying minimally rigid graph. 
Hence the upper bound on the number of embeddings given by the mixed volume 
might not be tight in every case.

\subsection{Henneberg I graphs}
For this simple class of Laman graphs the mixed volume bound is tight as we will demonstrate below. Our proof exploits the inductive structure of Henneberg I graphs which is why it cannot be used for Henneberg II graphs. 
\begin{lemma}\label{HennebergILemma}
For a Henneberg I graph on $n$ vertices, the mixed volume of system $\eqref{SubSoE}$ equals $2^{n-2}$.
\end{lemma}
\begin{proof}
Each Henneberg sequence starts with a triangle for which system \eqref{SubSoE} has mixed volume $2$. Starting from the triangle we consider a sequence of Henneberg I steps and show that the mixed volume doubles in each of these steps.

In a Henneberg I step we add one vertex $v_{n+1}$ and two edges $[v_r,v_{n+1}]$, $[v_q,v_{n+1}]$ with lengths $l_{r,n+1}$ and $l_{q,n+1}$. 
So our system of equations (\ref{SubSoE}) gets three new equations, namely
 \begin{eqnarray}
 s_{n+1} - x_{n+1}^2 - y_{n+1}^2 &=& 0 \label{eq_HI_1} \\
 s_r + s_{n+1}  - 2x_r x_{n+1} - 2y_r y_{n+1} - l_{r,n+1}^2 &=& 0 \label{eq_HI_2} \\
 s_q + s_{n+1}  - 2x_q x_{n+1} - 2y_q y_{n+1} - l_{q,n+1}^2 &=& 0. \label{eq_HI_3}
 \end{eqnarray}
In the new system of equations these three are the only polynomials involving $x_{n+1}$, $y_{n+1}$ and $s_{n+1}$, so Lemma~\ref{SeparationLemma} can be used to calculate the mixed volume separately. The projections of the Newton polytopes of equations \eqref{eq_HI_1}, \eqref{eq_HI_2} and \eqref{eq_HI_3} to the coordinates $x_{n+1}$, $y_{n+1}$ and $s_{n+1}$ are
\[
\text{conv}\left\{\begin{pmatrix}2&0&0\end{pmatrix}^T,  \begin{pmatrix}0&2&0\end{pmatrix}^T, \begin{pmatrix}0&0&1\end{pmatrix}^T \right\}  
\]
and twice
\[
\text{conv}\left\{\begin{pmatrix}1&0&0\end{pmatrix}^T,  \begin{pmatrix}0&1&0\end{pmatrix}^T, \begin{pmatrix}0&0&1\end{pmatrix}^T, \begin{pmatrix}0&0&0\end{pmatrix}^T \right\}\ .
\]
The mixed volume of these equals $2$. So by Lemma \ref{SeparationLemma} the mixed volume of the new system is twice the mixed volume of the system before the Henneberg I step.
\end{proof}
To get two new embeddings in every Henneberg I step we choose the new edge lengths to be almost equal to each other and much larger then all previous edge lengths (larger then the sum of all previous is certainly enough). 

\begin{cor}[Borcea and Streinu \cite{BorceaStreinu}] \label{TheoremHI}
The number of embeddings of Henneberg I graph frameworks is less than or equal to $2^{n-2}$ and this bound is sharp.
\end{cor}
Of course the elementary proof described in \cite{BorceaStreinu} of this statement does not need such heavy machinery as Bernstein's Theorem. The purpose of Lemma~\ref{HennebergILemma} is to show that the techniques described in this work apply here and that the BKK bound is tight in this case.

\subsection{Laman graphs on 6 vertices}\label{sec:6Vertices}
The first Laman graphs which are not constructable using only Henneberg I steps arise on $6$ vertices. A simple case analysis shows that up to isomorphisms there are only two such graphs, the Desargues graph and $K_{3,3}$ (see Figure~\ref{DesarguesAndK33}). 

\begin{figure}[h]
  \begin{center}
  \includegraphics[scale=0.18]{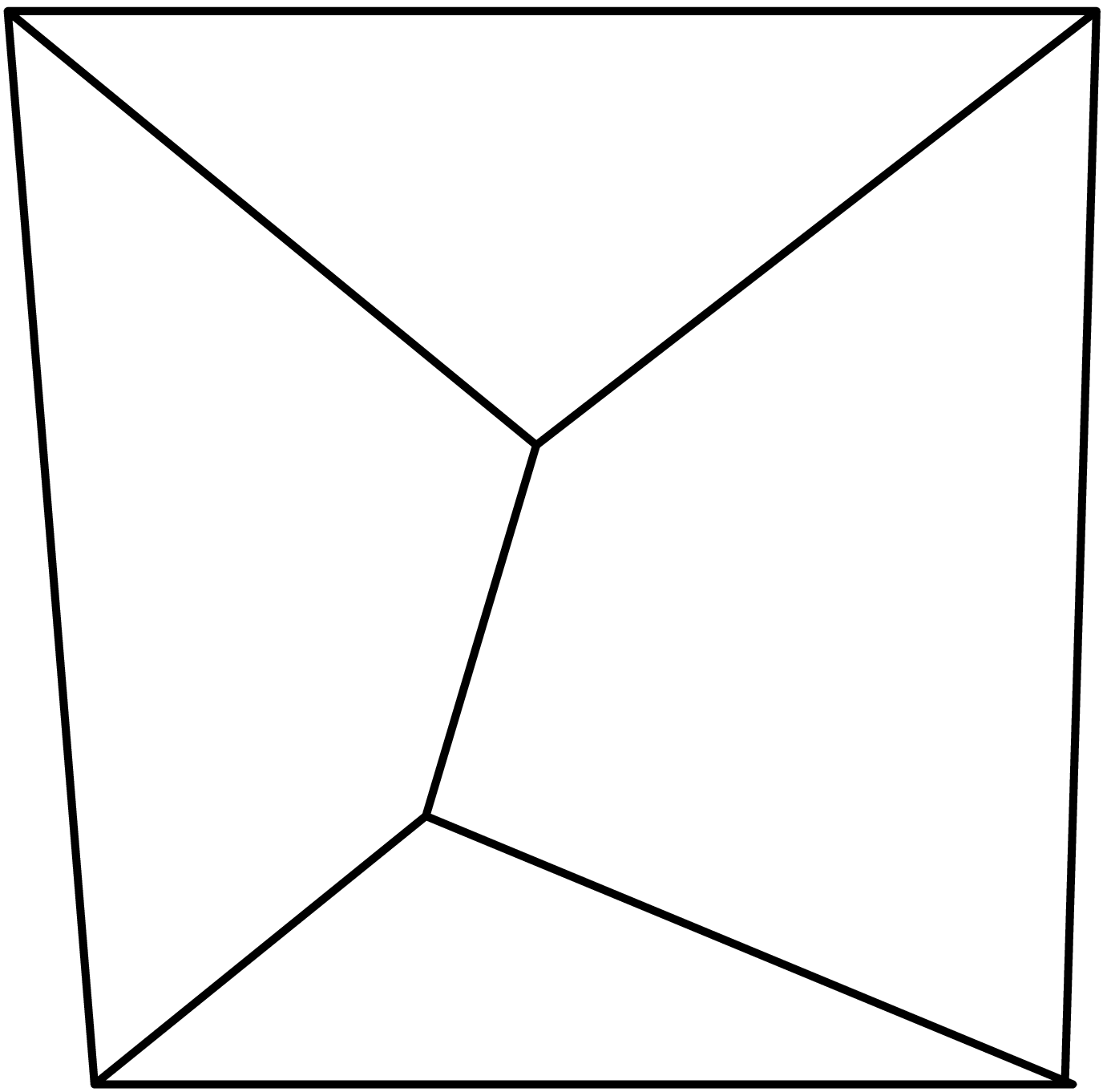} \qquad \includegraphics[scale=0.18]{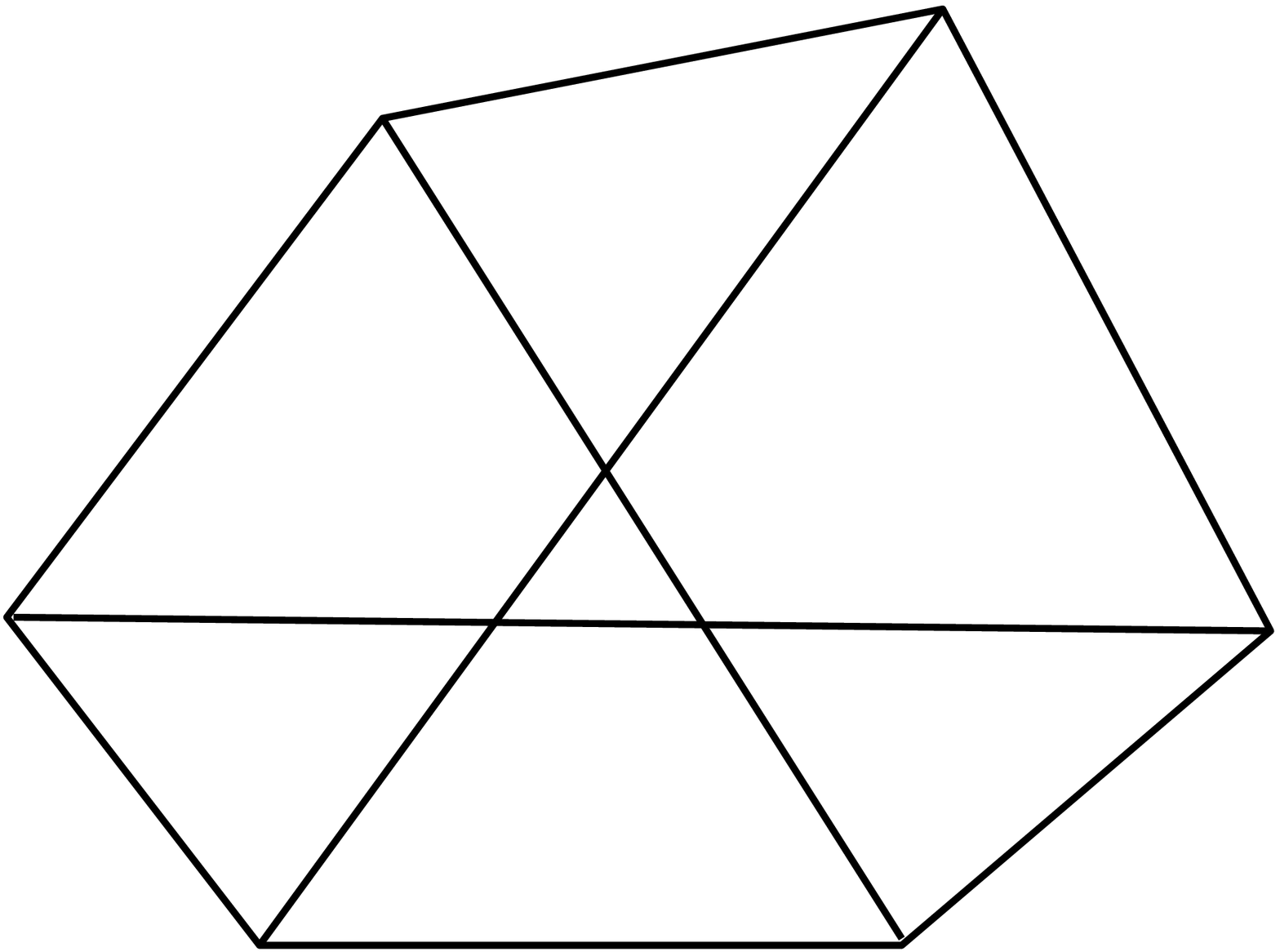}
  \end{center}
  \caption{Left: Desargues graph. Right: $K_{3,3}$.}
  \label{DesarguesAndK33}
\end{figure}

The number of embeddings of both graphs has been studied in detail. The Desargues graph is studied in \cite{BorceaStreinu} where the authors show that there can only be $24$ embeddings and that there exists a choice of edge lengths giving $24$ different embeddings. This is obtained by investigating the curve that is traced out by one of the vertices after one incident edge is removed. 
Husty and Walter \cite{HustyWalter} apply resultants to show that $K_{3,3}$ can have up to $16$  embeddings  and give as well specific edge lengths leading to $16$ different embeddings.\footnote{This corrects an earlier version of this paper.}

Both approaches rely on the special combinatorial structure of the specific graphs. The general bound in \cite{BorceaStreinu} for the number of embeddings of a graph with $6$ vertices yields $\binom{2\cdot(6-2)}{6-2}=70$. In this case the BKK bound gives a closer estimate. Namely the mixed volume of the system \eqref{SubSoE} (which uses the substitution trick to remove degeneracies) can be shown to be
 $32$  for  both graphs.\footnote{We used the PHCpack by Jan Verschelde for our mixed volume calculations, see \cite{phc}.}

\subsection{General case}
For the classes discussed above (Henneberg I, graphs on six vertices) as well as some other special cases,
the BKK bound on the number of embeddings resembles or even improves the general bound of $\binom{2n-4}{n-2} $. 
For the general case, the mixed volume approach for the system (\ref{SoE}) without the substitutions suggested by Emiris provides a simple,
but very weak bound. However, it may be of independent interest that  the mixed volume can be exactly determined as a function of $n$ and that in particular the value is independent of the structure of the Laman graph.  
\begin{theorem} \label{GeneralBound}
For any Laman graph on $n$ vertices, the mixed volume of the initial system \eqref{SoE} is exactly $4^{n-2}$.
\end{theorem}
\begin{proof}
The mixed volume of (\ref{SoE}) is at most the product of the degrees of the polynomial equations because it is less than or equal to the B\'{e}zout bound (see \cite{SturmfelsSolving}). To show that the mixed volume is at least this number we will use Lemma~\ref{LiftingLemma} to give a lifting that induces a mixed cell of volume $4^{n-2}$. 

For $i \in \{1, \dots, 4\}$ the Newton polytope $\NP(h_i)$ is a segment.
We claim that the polynomials $h_i$ can be ordered in a way such that  
for $i\geq 5$, $\NP(h_i)$ contains the edge $[0,2 \xi_i]$ where $\xi_i$ denotes the $i^{th}$ unit vector. 
To see this, note first that every polynomial $h_j$ ($1 \le j \le 2n$)
has a non-vanishing constant term and therefore $0 \in \NP(h_j)$.
For $i \in \{1, \dots, n\}$, each of the monomials $x^2_i$ and $y^2_i$ occurs in $h_j$ if and only if
the edge which is modeled by $h_j$ is incident to $v_i$.

Let $E' := E \setminus \{[v_1,v_2]\}$. The Henneberg construction of a Laman graph
allows to orient the edges such that in the graph
$(V,E')$ each vertex in $V \setminus \{v_1,v_2\}$ has exactly two incoming edges
(see \cite{JordanBerg,LeeStreinu}).
Namely, 
in a Henneberg I step the two new edges point to the new vertex. For a Henneberg II step we remember the direction of the deleted edge $\stackrel{\longrightarrow}{[v_r,v_s]}$ and let the new edge, which connects the new vertex to $v_s$, point to $v_s$. The other two new edges point to the new vertex.
(Figure~\ref{FigHennebergDirect} depicts this in an example.)

This orientation shows how to order the polynomials $h_5, \dots, h_{2n}$ in such a way
that the polynomials $h_{2i-1}$ and $h_{2i}$ model edges which are incoming edges of
the vertex $v_i$ within the directed graph. Remembering that the order of the variables was $(x_1,y_1,\dots,x_n,y_n)$ this implies that $2\xi_{2i-1}\in \NP(h_{2i-1})$ and $2\xi_{2i}\in \NP(h_{2i})$.
\begin{figure}[ht]
\setlength{\unitlength}{0.18pt}
	\begin{center}
		\begin{picture}(1200,550)(0,0)
			\put(10,40){\includegraphics[scale=0.18]{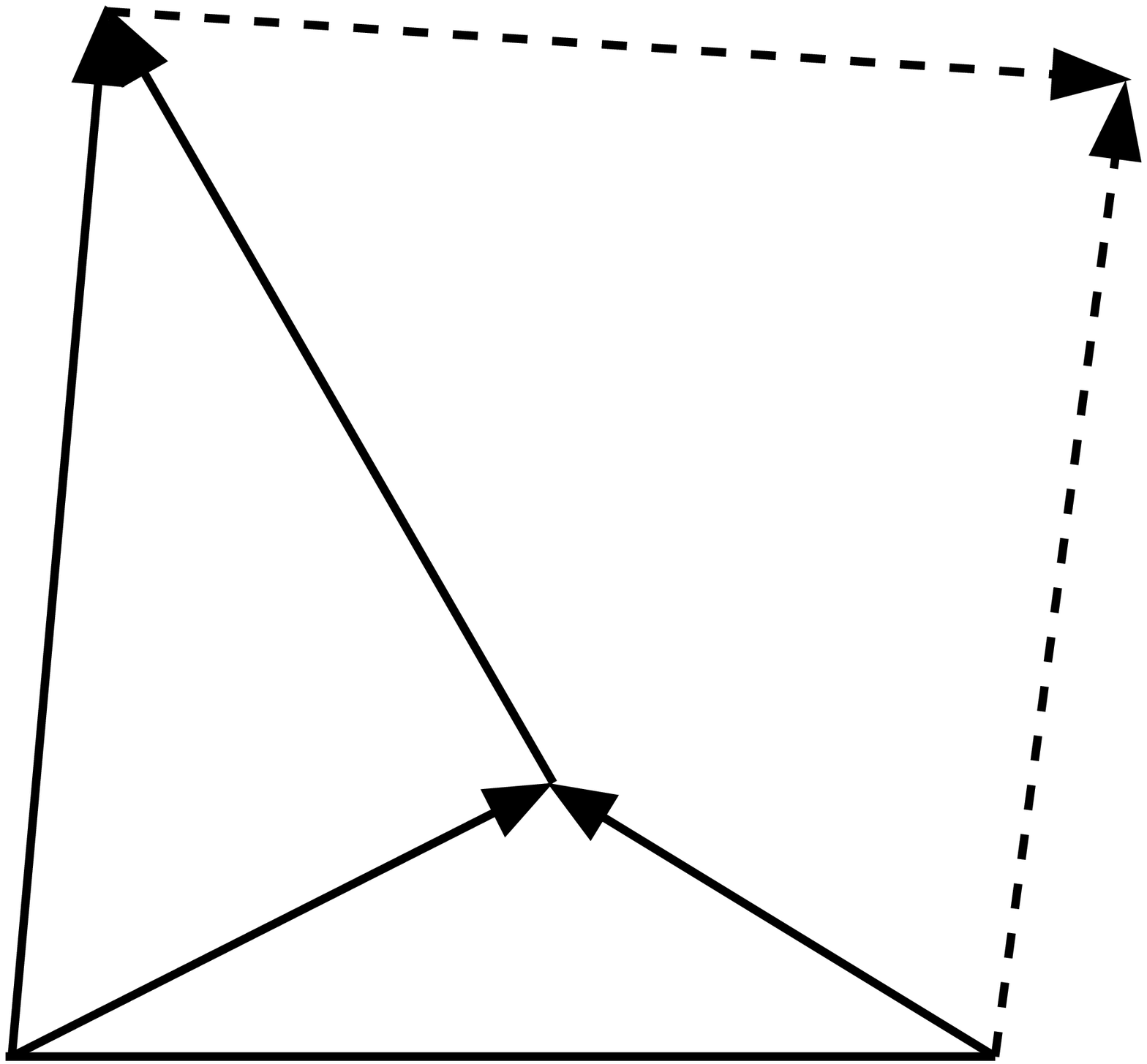}}
			\put(650,40){\includegraphics[scale=0.18]{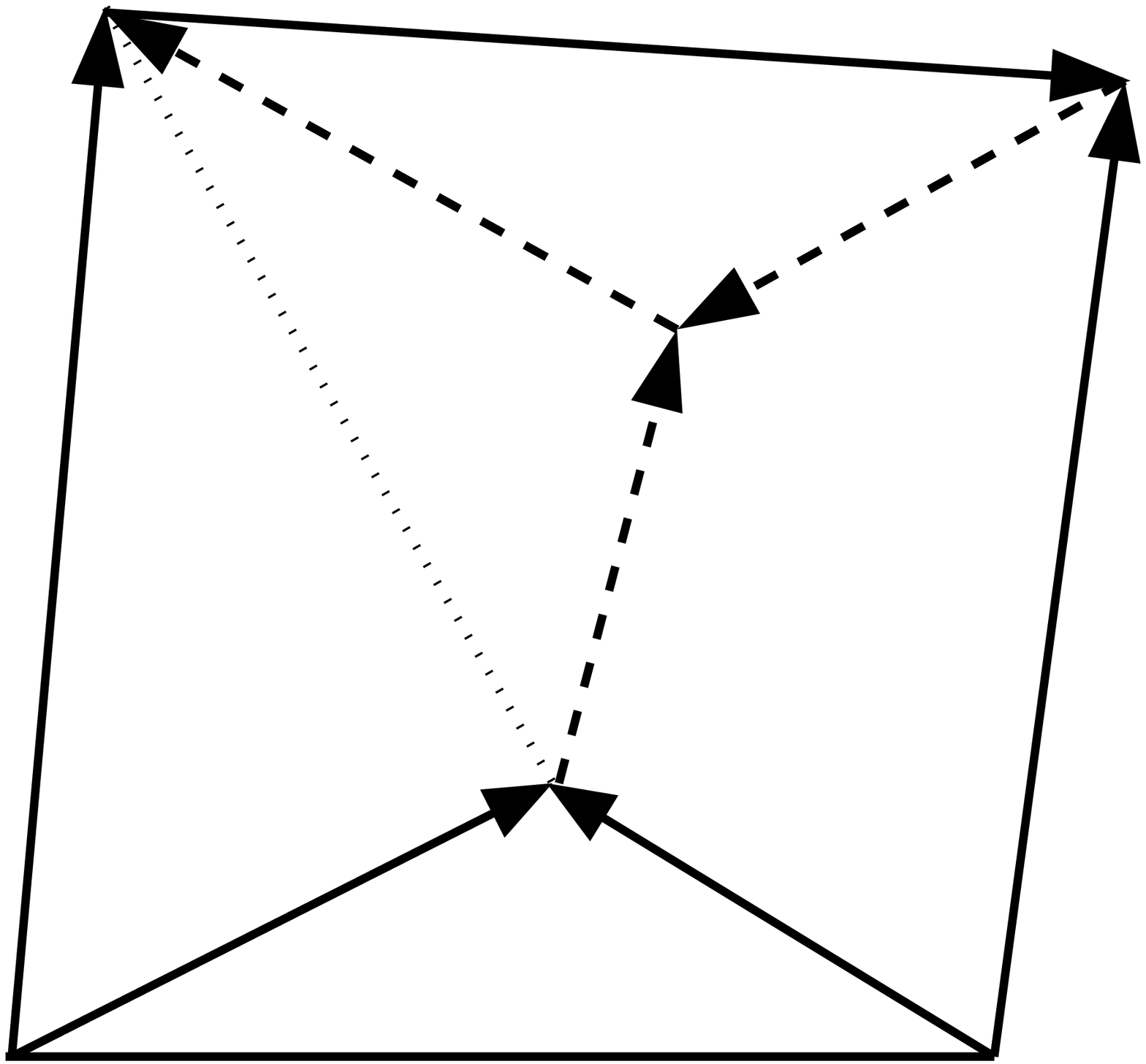}}
			\put(0,0){\mbox{$v_1$}}
			\put(450,0){\mbox{$v_2$}}
			\put(640,0){\mbox{$v_1$}}
			\put(1090,0){\mbox{$v_2$}}
			\put(290,165){\mbox{$v_3$}}
			\put(940,165){\mbox{$v_3$}}
			\put(0,540){\mbox{$v_4$}}
			\put(650,545){\mbox{$v_4$}}
			\put(520,510){\mbox{$v_5$}}
			\put(1170,510){\mbox{$v_5$}}
			\put(980,340){\mbox{$v_6$}}
		\end{picture}
	\caption{A Henneberg I and a Henneberg II step with directed edges.}
	\label{FigHennebergDirect}
\end{center}
\end{figure}
\setlength{\unitlength}{1pt}

 Now Lemma~\ref{LiftingLemma} can be used to describe a lifting that induces a subdivision that has 
$[\xi_1,0]+\dots+[\xi_4,0]+[2 \xi_5,0]+\dots+ [2 \xi_{2n},0] $ as a mixed cell. In the notation of Lemma~\ref{LiftingLemma} the chosen edges give rise to the edge matrix $E =\begin{pmatrix} \textup{Id}_4 & \mathbf{0}\\ \mathbf{0}& 2\textup{Id}_{2n-4} \end{pmatrix}$, where $\textup{Id}_k$ denotes the $k\times k$ identity matrix. Substituting this into the second condition (\ref{LiftingCondition}) of Lemma~\ref{LiftingLemma} we get that for each Newton polytope $\NP(h_i)$ all vertices $v^{(i)}_s$ of $\NP(h_i)$ which are not $0$ or $2\xi_i$ have to satisfy 
\[
\left( \mu_{1_1} -\mu_{i_1},\dots,\mu_ {{2n}_{2n}}-\mu_{i_{2n}} \right)\cdot v^{(i)}_s \leq 0 \, ,
\]
where we denote by $\mu_j=(\mu_{j_1},\dots,\mu_{j_{2n}})\in \Q^{2n}$ the lifting vector for $\NP(h_j)$. Since all the entries of each $v^{(i)}_s$ are non-negative this can easily be done by choosing the vectors $\mu_j$ such that their $j^{th}$ entry is sufficiently small and all other entries are sufficiently large. 
\end{proof}
The preliminary remarks at the beginning of this section further imply: 
\begin{cor}
The number of embeddings of a Laman graph framework with generic edge lengths is strictly less then $4^{n-2}$.
\end{cor}

\subsection{Open problems and future prospects}
Examples like the case study of Laman graph frameworks on $6$ vertices in Paragraph~\ref{sec:6Vertices} suggest that the mixed volume of the system \eqref{SubSoE} gives a significantly better bound on the number of embeddings than the one analyzed in Theorem~\ref{GeneralBound}. However it remains open to compute the mixed volume of the system \eqref{SubSoE} as a function of $n$ like  it was done for the system \eqref{SoE} in Theorem~\ref{GeneralBound}. 

The focus of our paper was on embeddings in the plane. See \cite{BorceaStreinu,EmirisVarvitsiotis} for embeddings into higher-dimensional spaces. With regard to the Bernstein bounds 
there are straightforward analogs of Lemma~\ref{HennebergILemma} and Theorem~\ref{GeneralBound} to higher dimensions.

\bibliography{Library}{}
\bibliographystyle{abbrv}

\end{document}